\documentclass[12pt]{article}

\usepackage{amsmath, epsfig, cite, lineno}
\usepackage{amssymb,amsthm}
\usepackage{amsfonts, color}
\usepackage{latexsym}
\usepackage{enumitem}

\newtheorem{thm}{Theorem}[section]

\newtheorem{cor}[thm]{Corollary}

\newtheorem{lem}[thm]{Lemma}

\def\N{{\mathbb N}}
\def\Z{{\mathbb Z}}

\numberwithin{equation}{section}

\begin{document}

\begin{center}
{\Large\bf A further $q$-generalization of the (C.2) and (G.2)\\[6pt] 
supercongruences of Van Hamme}
\end{center}

\vskip 2mm \centerline{Song-Xiao Li and Su-Dan Wang\thanks{Corresponding author}}
\begin{center}

{College of Mathematics Science, Inner Mongolia Normal University, Huhhot 010022, Inner Mongolia,\\
 People's Republic of China\\
 {\tt songxiao\_li@mails.imnu.edu.cn},\hspace{0.5em}{\tt sdwang@imnu.edu.cn}}

\end{center}

\vskip 0.7cm \noindent{\bf Abstract.} Applying the $q$-Zeilberger algorithm, we establish a unified $q$-analogue of the (C.2) and (G.2) supercongruences 
of Van Hamme, which can be viewed as a refinement of several previously known results. As consequences, we  obtain a $q$-analogue of supercongruence 
involving Bernoulli numbers, as well as a refinement of (G.2) supercongruence.
\vskip 3mm \noindent {\it Keywords}: $q$-Zeilberger algorithm; $q$-analogue; $q$-supercongruences; Bernoulli numbers.

\vskip 0.2cm \noindent{\it AMS Subject Classifications:} 33D15, 11B65, 05A10

\section{Introduction}
In 1997, Van Hamme\cite{Van} listed some interesting $p$-adic analogues of Ramanujan's and Ramanujan-type formulas, two of which are the supercongruences 
labeled {\rm (C.2)} and {\rm (G.2)} given below. 
\begin{align}
    \label{eq:Van-C2}\sum_{k=0}^{(p-1)/2}(4k+1)\frac{(\tfrac14)^4_k}{k!^4}&\equiv p\pmod{p^3},\\
    \label{eq:Van-G2}\sum_{k=0}^{(p-1)/4}(8k+1)\frac{(\tfrac12)^4_k}{k!^4}&\equiv p\frac{\Gamma_p(\tfrac12)\Gamma_p(\tfrac14)}{\Gamma_p(\tfrac34)}\pmod{p^3},
\end{align}
where $(a)_n$ is the Pochhammer symbol and the $p$-adic Gamma function $\Gamma_p$ is defined as
$\Gamma_p(0)=1$ and \[\Gamma_p(n)=(-1)^n\prod_{\substack{1\leqslant k < n\\ p\nmid k}}^{}k.\] Here and throughout this paper, $p$ is an odd prime.
Van Hamme\cite{Van} himself proved \eqref{eq:Van-C2} and another two supercongruences of his list.  
Long\cite{Long} further proved that \eqref{eq:Van-C2} holds modulo $p^4$ for $p>3$. Guo and the second author\cite{Guo-Wang-C2} gave a $q$-analogue of Long's result as follows:
for any positive odd integer $n$, modulo ${[n]\Phi_n(q)^3}$,
\begin{align}\label{eq:GW-C2}
    \sum_{k=0}^{(n-1)/2}[4k+1]\frac{(q;q^2)^4_k}{(q^2;q^2)^4_k}\equiv q^{(1-n)/2}[n]+\frac{(n^2-1)(1-q^2)}{24}q^{(1-n)/2}[n]^3 .
\end{align}
Here it is proper to familiarize ourselves with the standard $q$-hypergeometric notation: $[n]=(1-q^n)/(1-q)$ is the {\em $q$-integer}. The {\em $q$-shifted factorials} are defined by $(a;q)_0=1$
and $(a;q)_n=(1-a)(1-aq)\cdots(1-aq^{n-1})$ for $n=1,2,\dots$,with the compact notation $(a_1,\dots,a_m;q)_n=(a_1;q)_n\cdots(a_m;q)_n$ for products of some $q$-shifted factorials.
 Moreover, let $\Phi_n(q)$ be the $n$-th \emph{cyclotomic polynomial} in $q$, i.e.,$$\Phi_n(q)=\prod_{\substack{1\leqslant k \leqslant n\\ \gcd(k,n)=1}}^{}(q-\zeta^k),$$
where $\zeta$ is an $n$-th primitive root of unity. \par 
It follows easily from \eqref{eq:GW-C2} that, for any prime $p\geqslant5$ and integer $r\geqslant1$,
\begin{align}\label{eq:GW-gen}
    \sum_{k=0}^{(p^r-1)/2}(4k+1)\frac{(\tfrac14)^4_k}{k!^4}\equiv p^r \pmod{p^{r+3}},
\end{align} 
which was originally observed by Long\cite{Long}. By employing an approach similar to Zudilin\cite{Zu}, Wang and Hu\cite{Wang-Hu} proved the following generalization of \eqref{eq:GW-gen}:
\begin{align}\label{eq:WH}
    \sum_{k=0}^{(p^r-1)/2}(4k+1)\frac{(\frac{1}{2})_k^4}{k!^4}\equiv p^r+\frac{7}{6}p^{r+3}B_{p-3}\pmod{p^{r+4}},
\end{align}
where $B_{p-3}$ is the Bernoulli numbers of index $p-3$, thereby verifying a conjecture previously proposed by Guo\cite[Conjecture~6.2]{Guo-fourth}.\par
He\cite{He} proved that \eqref{eq:Van-G2} is true modulo the higher power $p^4$. An identical conclusion was also obtained by Swisher\cite{Swisher}, who further proved the following generalization of (1.2):
for $p\equiv 3\pmod4$ and $p\geqslant5$,
\begin{align}\label{eq:Swisher}
    \sum_{k=0}^{(3p-1)/4}(8k+1)\frac{(\tfrac14)^4_k}{(1)^4_k}\equiv \frac{-3p^2\Gamma_p(\tfrac12)\Gamma_p(\tfrac14)}{2\Gamma_p(\tfrac34)}\pmod{p^4}.
\end{align}
Liu and Wang\cite{LW-G2-p3} established a $q$-analogue of \eqref{eq:Van-G2}. Shortly afterwards, 
they\cite{Liu-Wang-G2} further proved the following $q$-supercongruence as a $q$-analogue of He's result : for a positive integer $n\equiv1 \pmod4$, modulo ${[n]\Phi_n(q)^3}$,
\begin{align}\label{eq:LW-G2}
    \sum_{k=0}^{M}[8k+1]\frac{(q;q^4)^4_k}{(q^4;q^4)^4_k}q^{2k}\equiv &\frac{(q^2;q^4)_{(n-1)/4}}{(q^4;q^4)_{(n-1)/4}}[n]q^{(1-n)/4}\nonumber\\
    &\times\left\{1+[n]^2\frac{(n^2-1)(1-q)^2}{24}
+ [n]^2\sum_{k=1}^{(n-1)/4}\frac{q^{4k-2}}{[4k-2]^2}\right\},
\end{align}
where $M=(n-1)/4$ or $n-1$.

Recently, applying Jackson's ${}_6\phi_5$ summation,
Guo~\cite{Guo-jack} obtained a unified form of the
$q$-supercongruences \eqref{eq:GW-C2} and \eqref{eq:LW-G2}.
More precisely, let $d,n,r$ be integers satisfying
$d \geqslant 2$ and $0 \leqslant n-r \leqslant dn-d$,
such that $d$ and $r$ are coprime and $n \equiv r \pmod{d}$.
Then, modulo $[n]\Phi_n(q)^3$,

\begin{align}\label{eq:Guo-uni}
    \sum_{k=0}^{M}&[2dk+r]\frac{(q^r;q^d)_k^4}{(q^d;q^d)^4_k}q^{(d-2r)k}
    \equiv [n]q^{r(r-n)/d}\frac{(q^{2r};q^d)_{(n-r)/d}}{(q^d;q^d)_{(n-r)/d}}\left(1-[n]^2\sum_{j=1}^{(n-r)/d}\frac{q^{dj}}{[dj]^2}\right),
\end{align}
where $M=(n-r)/d$ or $n-1$.\par
The first aim of this paper is to establish the following generalization of \eqref{eq:Guo-uni}.
\begin{thm}\label{thm:main-1}
Let $n$ and $d$ be positive integers satisfying $\gcd(n,d)=1$. Let $r$ be an integer with $n+d-nd\leqslant r \leqslant n$ and 
    $n\equiv r \pmod{d} $. Then, modulo $[n]\Phi_n(q)^4$,
    \begin{align*}
        &\sum_{k=0}^{M}[2dk+r]\frac{(q^r;q^d)_k^4}{(q^d;q^d)_k^4}q^{(d-2r)k}\nonumber\\
        &\hspace{2em}\equiv q^{r(r-n)/d}\frac{(q^{2r};q^d)_{(n-r)/d}}{(q^{d};q^d)_{(n-r)/d}}\Bigg([n]-([n]^3+[n]^4(1-q))\sum_{k=1}^{(n-r)/d}\frac{q^{kd}}{[kd]^2}\Bigg)\nonumber\\
        &\hspace{3em}+[n]^4\sum_{k=1}^{(n-r)/d}q^{(kd-d-2n+2r)(d-n)/d}\frac{(q^d;q^d)_{k-1}(q^{2r};q^d)_{k-1}}{[dk-d+r]^3(q^r;q^d)_{k-1}^2},
    \end{align*}
    where $M=(n-r)/d$ or $n-1$.
\end{thm}
For $d=2,r=1$, Theorem~\ref{thm:main-1} can be stated as follows.
\begin{cor}
Let $n$ be a positive odd integer. Then, modulo $[n]\Phi_n(q)^4$,
    \begin{align}\label{eq:d=2}
        \sum_{k=0}^{M}[4k+1]\frac{(q;q^2)_k^4}{(q^2;q^2)_k^4}
        &\equiv q^{(1-n)/2}\Bigg([n]-([n]^3+[n]^4(1-q))\sum_{k=1}^{(n-1)/2}\frac{q^{2k}}{[2k]^2}\Bigg)\nonumber\\
        &\hspace{1em}+[n]^4\sum_{k=1}^{(n-1)/2}q^{2k}\frac{(q^2;q^2)_{k-1}^2}{[2k-1]^3(q;q^2)_{k-1}^2},
    \end{align}
    where $M=(n-1)/2$ or $n-1$.
\end{cor} 
Throughout the paper, the harmonic numbers of order $r$ are defined by
\[H^{(r)}_m=\sum_{k=1}^{m}\frac{1}{k^r} \hspace{1em}\text{with}\hspace{1em} r,m\in \Z^+.\]
Letting $q\to 1$ and $n=p$ in \eqref{eq:d=2}, and making use of the congruences
\[
\sum_{k=1}^{(p-1)/2}\frac{16^{k-1}}{(2k-1)^3\binom{2k-2}{k-1}^2}
\equiv \frac{7}{4}B_{p-3}\pmod{p},
\]
and for $p\geqslant 5$,
\[
H_{\frac{p-1}{2}}^{(2)}\equiv \frac{7}{3}pB_{p-3}\pmod{p^2},
\]
proposed by Z.-W.~Sun~\cite{Sun} and Z.-H.~Sun~\cite[Corollary~5.2]{Sunzh}, respectively.
We immediately obtain the case $r=1$ of \eqref{eq:WH}. Therefore, the $q$-supercongruence 
\eqref{eq:d=2} is a partial $q$-analogue of \eqref{eq:WH}.

Taking $q\to1$, $d=4$, $r=1$, and $n=p$ in Theorem~\ref{thm:main-1}, we obtain a result on the (G.2) supercongruence modulo $p^5$.
\begin{cor}
Let $p\equiv 1\pmod4$ be a prime. Then, modulo $p^5$,
    \begin{align*}
\sum_{k=0}^{M}(8k+1)\frac{\big(\tfrac14\big)_k^{4}}{(1)_k^{4}
}
\equiv
\frac{\big(\tfrac12\big)_{(p-1)/4}}{(1)_{(p-1)/4}}
\left(
p-\frac{p^{3}}{16}
H_{\frac{p-1}{4}}^{(2)}
\right)
+p^{4}
\sum_{k=1}^{(p-1)/4}
\frac{(1)_{k-1}\big(\tfrac12\big)_{k-1}}
{(4k-3)^{3}\big(\tfrac14\big)_{k-1}^{2}},
\end{align*}
where $M=(p-1)/4$ or $p-1$.
\end{cor} 
Note that Guo and Schlosser\cite{Guo-Scho} gave a generalization of \eqref{eq:Swisher} as follows:
Let $d,n,r$ be integers satisfying $d\geqslant2, r\leqslant d-2$, and $n\geqslant d-r$, such that $d$ and 
$r$ are coprime, and $n\equiv -r\pmod{d}$. Then
\begin{align}\label{eq:GS}
    \sum_{k=0}^{M}&[2dk+r]\frac{(q^r;q^d)^4_k}{(q^d;q^d)^4_k}q^{(d-2r)k}\nonumber\\
    &\equiv \begin{cases}
        0 \pmod{[n]\Phi_n(q)^3},   & \text{if} \hspace{0.5em}d=2,\\[6pt]
        \displaystyle
        q^{r(n+r-dn)/d}\frac{(q^{2r};q^d)_{((d-1)n-r)/d}}{(q^d;q^d)_{((d-1)n-r)/d}}[(d-1)n]\pmod{[n]\Phi_n(q)^3}, & \text{if} \hspace{0.5em}d\geqslant3,
    \end{cases}
\end{align}where $M=((d-1)n-r)/d$ or $n-1$.\par
Another objective of this paper is to provide the following generalization of \eqref{eq:GS}.
\begin{thm}\label{thm:main-2}
    Let $n$ and $d\geqslant2$ be positive integers satisfying $\gcd(n,d)=1$. Let $r$ be an integer with $d-n\leqslant r \leqslant dn-n$ and 
    $n\equiv -r \pmod{d} $. Then, modulo $[n]\Phi_{n}(q)^4$,
    \begin{align*}
        &\sum_{k=0}^{M}[2dk+r]\frac{(q^r;q^d)_k^4}{(q^d;q^d)_k^4}q^{(d-2r)k}\nonumber\\
        &\hspace{2em}\equiv q^{r(n+r-dn)/d}\frac{(q^{2r};q^d)_{((d-1)n-r)/d}}{(q^{d};q^d)_{((d-1)n-r)/d}}\nonumber\\
        &\hspace{3em}\times\Bigg([(d-1)n]-([(d-1)n]^3+[(d-1)n]^4(1-q))\sum_{k=1}^{((d-1)n-r)/d}\frac{q^{kd}}{[kd]^2}\Bigg)\nonumber\\
        &\hspace{3em}+[(d-1)n]^4\sum_{k=1}^{((d-1)n-r)/d}q^{(kd-d-2dn+2n+2r)(d-dn+n)/d}\frac{(q^d;q^d)_{k-1}(q^{2r};q^d)_{k-1}}{[dk-d+r]^3(q^r;q^d)_{k-1}^2},
    \end{align*}
    where $M=((d-1)n-r)/d$ or $n-1$.
\end{thm}
Substituting $d=4$, $r=1$, and $n=p$ into Theorem~\ref{thm:main-2} and letting $q\to1$,
we strengthen the supercongruence \eqref{eq:Swisher} to modulo $p^5$.
\begin{cor}
    Let $p\equiv 3\pmod4$ be a prime. Then, modulo $p^5$
\begin{align}\label{eq:dn-d=4}
\sum_{k=0}^{M}(8k+1)\frac{\big(\tfrac14\big)_k^{4}}{(1)_k^{4}}
&\equiv
\frac{\big(\tfrac12\big)_{(3p-1)/4}}{(1)_{(3p-1)/4}}
\left(
3p-\frac{27p^3}{16}\,H_{(3p-1)/4}^{(2)}
\right)
+81p^4
\sum_{k=1}^{(3p-1)/4}
\frac{(1)_{k-1}\big(\tfrac12\big)_{k-1}}
{(4k-3)^3\big(\tfrac14\big)_{k-1}^2},
\end{align}
where $M=(3p-1)/4$ or $p-1$.
\end{cor}
Comparing \eqref{eq:Swisher} with \eqref{eq:dn-d=4}, and utilizing 
the fact that $\Gamma_p(\tfrac14)\Gamma_p(\tfrac34)=(-1)^{(p+1)/4}$ 
for $p\equiv 3 \pmod4$, we derive the following result.
\begin{cor}
    Let $p\equiv3\pmod4$ be a prime with $p\geqslant5$. Then,
    \begin{align*}
        \frac{\big(\tfrac12\big)_{(3p-1)/4}}{(1)_{(3p-1)/4}}
\left(
3p-\frac{27p^3}{16}\,H_{(3p-1)/4}^{(2)}
\right)\equiv (-1)^{(p-3)/4}\frac32p^2\Gamma_p(\tfrac14)^2\Gamma_p(\tfrac12)\pmod{p^4}.
    \end{align*}
\end{cor}

\section{Preliminary lemmas}
In this section, we present several necessary lemmas that will be used in the subsequent proofs.
We will make essential use of two results established by Xu and Wang \cite{XW} in our proof. 
\begin{lem}
    Let $d,n$ be positive integers and $r$ be an integer with $n\geqslant r\geqslant n-dn+d$, such that $\gcd(n,d)=1$, 
    and $n\equiv r \pmod{d}$. Then, modulo $[n]\Phi_n(q)(1-aq^n)(a-q^n)(b-q^n)$,
    \begin{align*}
        \sum_{k=(n-r+d)/d}^{n-1}[2dk+r]\frac{(aq^r,q^r/a,q^r/b,q^r;q^d)_k}{(aq^d,q^d/a,bq^d,q^d;q^d)_k}b^kq^{(d-2r)k}\equiv0.
    \end{align*}
\end{lem}
\begin{lem}
    Let $d,n$ be positive integers and $r$ be an integer with $dn-n\geqslant r\geqslant d-n$, such that $\gcd(n,d)=1$, 
    and $n\equiv -r \pmod{d}$. Then, modulo $[n]\Phi_n(q)(1-aq^{dn-n})(a-q^{dn-n})(b-q^{dn-n})$,
    \begin{align*}
        \sum_{k=((d-1)n-r+d)/d}^{n-1}[2dk+r]\frac{(aq^r,q^r/a,q^r/b,q^r;q^d)_k}{(aq^d,q^d/a,bq^d,q^d;q^d)_k}b^kq^{(d-2r)k}\equiv0.
    \end{align*}
\end{lem}
In what follows, we give some statements that are fundamental to the proof of Theorem~\ref{thm:main-1},
one of which was proved by Guo~\cite[Lemma~2.1]{Guo-qE2F2}.
\begin{lem}\label{lem:n-coprime}
Let $n$ and $d$ be positive integers, and let $r$ be an integer with $r<n$ and $n\equiv r\pmod{d}$. Write
\[\frac{{(q^r;q^d)}_{(n-r)/d}}{{(q^d;q^d)}_{(n-r)/d}}=\frac{A(q)}{B(q)},\]
where $A(q)$ and $B(q)$ are relatively prime polynomials in $q$. Then $B(q)$ is relatively prime to $1-q^n$.
\end{lem}
We now turn to another statement, which generalizes a result of Guo~\cite[Lemma~2.3]{Guo-qE2F2}.
\begin{lem}\label{lem:F-n}
    Let $n$ and $d$ be positive integers satisfy $\gcd(n,d)=1$ . Let $r$ be an integer with $n+d-nd\leqslant r \leqslant n$ and $n\equiv r\pmod{d}$. Then modulo $[n]\Phi_n(q)^4$,
    \begin{align}\label{eq:F-n}
        &q^{(n-r)(d-n-3r)/2d}[2n-r]\frac{(q^r;q^d)_{2(n-r)/d}}{(q^d;q^d)^2_{(n-r)/d}}\nonumber\\
        &\hspace{2em}\equiv(-1)^{(n-r)/d} q^{r(r-n)/d}\left\{[n]-([n]^3+[n]^4(1-q))\sum_{k=1}^{(n-r)/d}\frac{q^{kd}}{[kd]^2}\right\}.
    \end{align}
\end{lem}
\begin{proof}
    It is readily seen that the left-hand side of \eqref{eq:F-n} can be written as
    \[q^{(n-r)(d-n-3r)/2d}[n]\frac{(q^r;q^d)_{(n-r)/d}(q^{n+d};q^d)_{(n-r)/d}}{(q^d;q^d)_{(n-r)/d}^2}.\]
    Note that $(q^{n+d};q^d)_{(n-r)/d}\equiv (q^d;q^d)_{(n-r)/d}\pmod{[n]}$. By Lemma~\ref{lem:n-coprime}, to prove \eqref{eq:F-n}, it is enough to 
    show that 
    \begin{align}\label{eq:F-mid}
        &\frac{(q^r;q^d)_{(n-r)/d}(q^{n+d};q^d)_{(n-r)/d}}{(q^d;q^d)_{(n-r)/d}^2}\nonumber\\
        &\hspace{0.5em}\equiv (-1)^{(n-r)/d}q^{(n-r)(n-d+r)/2d}\left\{1-([n]^2+[n]^3(1-q))\sum_{k=1}^{(n-r)/d}\frac{q^{kd}}{[kd]^2}\right\}\pmod{\Phi_n(q)^4}.
    \end{align}
    We have the equation
    \[(1-q^{n-kd})(1-q^{n+kd})+(1-q^{kd})^2q^{n-kd}-q^n(1-q^n)^2-q^n(1-q^n)^3=(1-q^n)^4,\]
    it follows that 
    \begin{align*}
        (1-q^{n-kd})(1-q^{n+kd})\equiv-(1-q^{kd})^2q^{n-kd}+q^n(1-q^n)^2+q^n(1-q^n)^3\pmod{\Phi_n(q)^4}.
    \end{align*}
    Moreover, for $k=1,2,\dots,\frac{n-r}{d}$, the polynomial $1-q^{kd}$ is relatively prime to $\Phi_n(q)$. Then
    \begin{align*}
        &\frac{(q^r;q^d)_{(n-r)/d}(q^{n+d};q^d)_{(n-r)/d}}{(q^d;q^d)_{(n-r)/d}^2}\\
        &\hspace{0.5em}=\prod_{k=1}^{(n-r)/d}\frac{(1-q^{n-kd})(1-q^{n+kd})}{(1-q^{kd})^2}\\
        &\hspace{0.5em}\equiv(-1)^{(n-r)/d}q^{(n-r)(n-d+r)/2d}\left\{1-([n]^2+[n]^3(1-q))\sum_{k=1}^{(n-r)/d}\frac{q^{kd}}{[kd]^2}\right\}\pmod{\Phi_n(q)^4}.
    \end{align*}      
    This prove \eqref{eq:F-mid}.
\end{proof}

\begin{lem}\label{lem:G-div-n}
    Let $n$ and $d$ be positive integers satisfying $\gcd(n,d)=1$. Let $r$ be an integer with $n+d-nd\leqslant r \leqslant n$ and 
    $n\equiv r \pmod{d} $. For $k\in \{1,2,\dots,\frac{n-r}{d}\}$, write 
    \[\frac{(1-q^n)(q^d;q^d)_{k-1}(q^{2r};q^d)_{k-1}}{(q^r;q^d)_{k}(q^r;q^d)_{k-1}}=\frac{A_k(q)}{B_k(q)},\]
    where $A_k(q)$ and $B_k(q)$ are relatively prime polynomials in $q$. Then $B_k(q)$ is relatively prime to $1-q^n$.
\end{lem}
\begin{proof}
    Clearly,
    \begin{align*}
        q^m-1=
        \begin{cases}
            \prod_{t\mid m}^{}\Phi_t(q),     & \text{if}\hspace{0.5em} m>0,\\[4pt]
            -q^m\prod_{t\mid m}^{}\Phi_t(q), & \text{if}\hspace{0.5em} m<0.
        \end{cases}
    \end{align*}
    Therefore, we can write
    \begin{align*}
        (q^d;q^d)_{k-1}&=\prod_{t=1}^{\infty}\Phi_t(q)^{f_t},\\
        (q^{2r};q^d)_{k-1}&=\pm q^u\prod_{t=1}^{\infty}\Phi_t(q)^{g_t},\\
        (q^{r};q^d)_{k}&=\pm q^v\prod_{t=1}^{\infty}\Phi_t(q)^{h_t},\\
        (q^{r};q^d)_{k-1}&=\pm q^w\prod_{t=1}^{\infty}\Phi_t(q)^{l_t},
    \end{align*}
where $u,v,w$ are integers, and $f_t, g_t,h_t$ and $l_t$ denote the numbers of multiples of $t$
in the sets $\{d,2d, \dots ,(k-1)d \}$, $\{2r, 2r+d,\dots ,2r+(k-2)d\}$, $\{r,r+d,\dots, r+(k-1)d\}$ 
and $\{r,r+d,\dots, r+(k-2)d\}$, respectively.
\par We now suppose that $t\mid n$. Since gcd$(n,d)$=1, we have gcd$(d,t)$=1. It follows that 
one and only one integer in the $S_k=\{r+ktd,r+(kt+1)d,\dots, r+(kt+t-1)d\}$ is divisible by $t$ for $k\in \N$.
First consider the case of $t\mid r$, we have $f_t=\lfloor \frac{k-1}{t}\rfloor$ and $h_t=\lfloor \frac{k-1}{t}\rfloor+1$, where $\lfloor x\rfloor$ stands for 
the greatest integer not exceeding $x$. Meanwhile, we also have $g_t=l_t=\lfloor \frac{k-2}{t}\rfloor+1$. Therefore, 
$f_t+g_t-h_t-l_t= -1$. Now assume that $t\nmid r$, we have $f_t=\lfloor \frac{k-1}{t}\rfloor$ , $h_t\leqslant
\lfloor \frac{k-1}{t}\rfloor+1$, and 
$g_t,l_t\leqslant \lfloor \frac{k-2}{t}\rfloor+1$. Thus, $f_t+g_t-h_t-l_t\geqslant -1$.\par
In view of above argument, we finish the proof.
\end{proof}

\begin{lem}\label{lem:q-integers-n}
     Let $n$ and $d$ be positive integers satisfying $\gcd(n,d)=1$. Let $r$ be an integer with $n+d-nd\leqslant r \leqslant n$ and 
    $n\equiv r \pmod{d} $. For $k\in \{1,2,\dots,\frac{n-r}{d}\}$, write 
    \[\frac{1-q^{n}}{1-q^{dk}}=\frac{C_k(q)}{D_k(q)},\] where $C_k(q)$ and $D_k(q)$ are relatively prime polynomials in $q$. 
    Then $D_k(q)$ is relatively prime to $1-q^n$.
\end{lem}
\begin{proof}
    We immediately obtain the desired result by using the fact that $q^{dk}-1$ can be written as a product of different 
    cyclotomic polynomials. 
\end{proof}

\section{Proof of Theorem~\ref{thm:main-1}}
For fixed integers $d$ and $r$, we consider the following rational functions in $q^m$ and $q^k$:
\begin{align*}
    &F(m,k)=(-1)^kq^{dk(k-2m-1)/2+(d-2r)m}\frac{[2dm+r](q^r;q^d)^3_m(q^r;q^d)_{m+k}}{(q^d;q^d)^3_m(q^d;q^d)_{m-k}(q^r;q^d)^2_k},\\
    &G(m,k)=(-1)^{k-1}q^{dk(k-2m+1)/2+(d-2r)(m-1)}\frac{(q^r;q^d)^3_m(q^r;q^d)_{m+k-1}}{(1-q)^2(q^d;q^d)^3_{m-1}(q^d;q^d)_{m-k}(q^r;q^d)^2_k},
\end{align*}
where we assume that $1/(q^d;q^d)_m=0$ for any negative integer $m$. 
It is straightforward to verify that
\begin{align}\label{eq:qz}
    [dk-d+r]F(m,k-1)-[dk-d+2r]F(m,k)=G(m+1,k)-G(m,k).
\end{align}
Firstly, summing over $m$ from 0 to $(n-r)/d$, we obtain
\begin{align*}
    \displaystyle
    \sum_{m=0}^{(n-r)/d}[dk-d+r]F(m,k-1)-\sum_{m=0}^{(n-r)/d}[dk-d+2r]F(m,k)=G\left(\frac{n+d-r}{d},k\right),
\end{align*}
where we have used $G(0,k)=0$. Then, summing over $k$ from $1$ to $(n-r)/d$ and performing an iteration, observing that $F(m,(n-r)/d)=0$ for $m<(n-r)/d$,  we get
\begin{align}\label{eq:FG-n}
    \sum_{m=0}^{(n-r)/d}F(m,0)=&F\left(\frac{n-r}{d},\frac{n-r}{d}\right)\prod_{j=1}^{(n-r)/d}\frac{[dj-d+2r]}{[dj-d+r]}\nonumber\\
    &+\sum_{k=1}^{(n-r)/d}G\left(\frac{n+d-r}{d},k\right)\prod_{j=1}^{k-1}\frac{[dj-d+2r]}{[dj-d+r]}\frac{1}{[dk-d+r]}.
\end{align}
By Lemma~\ref{lem:F-n}, we conclude that, modulo $[n]\Phi_n(q)^4$,
\begin{align}\label{eq:F-final}
    &F\left(\frac{n-r}{d},\frac{n-r}{d}\right)\prod_{j=1}^{(n-r)/d}\frac{[dj-d+2r]}{[dj-d+r]}\nonumber\\
    &\hspace{1em}=(-1)^{(n-r)/d}q^{(n-r)(d-n-3r)/2d}[2n-r]\frac{(q^r;q^d)_{2(n-r)/d}}{(q^d;q^d)^2_{(n-r)/d}}\cdot
    \frac{(q^{2r};q^d)_{(n-r)/d}}{(q^d;q^d)_{(n-r)/d}}\nonumber\\
    &\hspace{1em}\equiv q^{r(r-n)/d}\frac{(q^{2r};q^d)_{(n-r)/d}}{(q^d;q^d)_{(n-r)/d}}\left\{[n]-([n]^3+[n]^4(1-q))
    \sum_{k=1}^{(n-r)/d}\frac{q^{kd}}{[kd]^2}\right\}.
\end{align}
For $k=1,2,\dots,(n-r)/d$, we see that 
\begin{align}\label{eq:G}
    &G\left(\frac{n+d-r}{d},k\right)\nonumber\\
    &\hspace{0.5em}=(-1)^{k-1}q^{d\binom{k}{2}-(n-r)(k-(d-2r)/d)}\frac{(q^r;q^d)^3_{(n-r)/d+1}(q^r;q^d)_{(n-r)/d+k}}
    {(1-q)^2(q^d;q^d)^3_{(n-r)/d}(q^d;q^d)_{(n-r)/d-k+1}(q^r;q^d)_k^2}\nonumber\\
    &\hspace{0.5em}=(-1)^{k-1}q^{d\binom{k}{2}-(n-r)(k-(d-2r)/d)}\frac{(1-q^n)^4(q^r;q^d)^4_{(n-r)/d}(q^{n+d};q^d)_{k-1}}
    {(1-q)^2(q^d;q^d)^3_{(n-r)/d}(q^d;q^d)_{(n-r)/d-k+1}(q^r;q^d)^2_k}.
\end{align}
Since $q^n\equiv 1\pmod{\Phi_n(q)}$, it follows that 
\begin{align*}
    (q^d;q^d)_{(n-r)/d-k+1}&=\frac{(q^d;q^d)_{(n-r)/d}}{(q^{n-r-dk-2d};q^d)_{k-1}}\nonumber\\
    &\equiv \frac{(q^d;q^d)_{(n-r)/d}}{(q^{-r-dk-2d};q^d)_{k-1}}\nonumber\\
    &=(-1)^{k-1}q^{(k-1)(dk+2r-2d)/2}\frac{(q^d;q^d)_{(n-r)/d}}{(q^{r};q^d)_{k-1}}\pmod{\Phi_n(q)}
\end{align*}
and 
\begin{align*}
    \frac{(q^r;q^d)_{(n-r)/d}}{(q^{d};q^d)_{(n-r)/d}}&=\prod_{j=1}^{(n-r)/d}\frac{1-q^{dj+r-d}}{1-q^{n-dj-r+d}}\nonumber\\
    &\equiv \prod_{j=1}^{(n-r)/d}\frac{1-q^{dj+r-d}}{1-q^{d-r-dj}}=(-1)^{(n-r)/d}q^{(n-r)(n+r-d)/2d}\pmod{\Phi_n(q)}.
\end{align*}
Applying the above two $q$-congruences, we deduce from \eqref{eq:G} that, for $1\leqslant k \leqslant (n-r)/d$,
\begin{align*}
    G\left(\frac{n+d-r}{d},k\right)\equiv q^{(kd-d-2n+2r)(d-n)/d}\frac{[n]^4(q^d;q^d)_{k}}{[dk][dk-d+r](q^r;q^d)_{k}}\pmod{\Phi_n(q)^5}.
\end{align*}
Furthermore, we obtain
\begin{align}\label{eq:G-final}
    &G\left(\frac{n+d-r}{d},k\right)\prod_{j=1}^{k-1}\frac{[dj-d+2r]}{[dj-d+r]}\cdot\frac{1}{[dk-d+r]}\nonumber\\
    &\hspace{1em}\equiv q^{(kd-d-2n+2r)(d-n)/d}\frac{[n]^4(q^d;q^d)_{k}(q^{2r};q^d)_{k-1}}{[dk][dk-d+r]^2 (q^r;q^d)_k(q^r;q^d)_{k-1}}\pmod{\Phi_n(q)^5}.
\end{align}
In terms of Lemma~\ref{lem:G-div-n} and Lemma~\ref{lem:q-integers-n}, we conclude that the right-hand side of \eqref{eq:G-final} is congruent to 0 modulo $[n]$.
Via similar arguments as the ones in the proof of Lemma~\ref{lem:G-div-n} and Lemma~\ref{lem:q-integers-n}, it is easy to see $G(\frac{n+d-r}{d},k)\equiv0 \pmod{[n]}$
from \eqref{eq:G}, and so is the left-hand side of \eqref{eq:G-final}. Since the least common mutiple of $[n]$ and $\Phi_n(q)^5$ is $[n]\Phi_n(q)^4$, the $q$-congruence
\eqref{eq:G-final} holds modulo $[n]\Phi_n(q)^4$. \par
Combining \eqref{eq:FG-n}, \eqref{eq:F-final} and the case modulo $[n]\Phi_n(q)^4$ of \eqref{eq:G-final}, and using the identity
\begin{align}\label{eq:iden}
    \frac{(q^d;q^d)_{k}(q^{2r};q^d)_{k-1}}{[dk][dk-d+r]^2 (q^r;q^d)_k(q^r;q^d)_{k-1}}=
    \frac{(q^d;q^d)_{k-1}(q^{2r};q^d)_{k-1}}{[dk-d+r]^3(q^r;q^d)^2_{k-1}},
\end{align}
we complete the proof.

\section{Proof of Theorem~\ref{thm:main-2}}
We first intrduce a result analogous to Lemma~\ref{lem:n-coprime} proposed by Guo\cite[Lemma~4.1]{Guo-qE2F2}.
\begin{lem}
    Let $n$ and $d$ be positive integers. Let $r$ be an integer with $\gcd(r,d)=1$, $r<(d-1)n$, and $d|(n+r)$. write 
    \[\frac{(q^r;q^d)_{((d-1)n-r)/d}}{(q^d;q^d)_{((d-1)n-r)/d}}=\frac{A(q)}{B(q)},\]
    where $A(q)$ and $B(q)$ are relatively prime polynomials in $q$. Then $B(q)$ is relatively prime to $1-q^n$.
\end{lem}
In terms of the same technique as in the proof of Lemma~\ref{lem:F-n} and the relation
\[[2(d-1)n-r]\frac{(q^r;q^d)_{2((d-1)n-r)/d}}{(q^r;q^d)^2_{((d-1)n-r)/d}}=[(d-1)n]\frac{(q^r;q^d)_{((d-1)n-r)/d}(q^{(d-1)n+d};q^d)_{((d-1)n-r)/d}}
{(q^d;q^d)^2_{((d-1)n-r)/d}},\]
we can prove the following result.
\begin{lem}\label{lem:F-dn}
    Let $n$ and $d\geqslant2$ be positive integers satisfy $\gcd(n,d)=1$. Let $r$ be an integer with $d-n\leqslant r\leqslant (d-1)n$, and $n\equiv -r\pmod{d}$.
    Then modulo $[n]\Phi_n(q)^4$,
    \begin{align*}\label{eq:F-dn}
        &q^{((d-1)n-r)(d-(d-1)n-3r)/2d}[2(d-1)n-r]\frac{(q^r;q^d)_{2((d-1)n-r)/d}}{(q^d;q^d)^2_{((d-1)n-r)/d}}\nonumber\\
        &\hspace{2em}\equiv(-1)^{((d-1)n-r)/d} q^{r(n+r-dn)/d}\nonumber\\
        &\hspace{3em}\times\left\{[(d-1)n]-([(d-1)n]^3+[(d-1)n]^4(1-q))\sum_{k=1}^{((d-1)n-r)/d}\frac{q^{kd}}{[kd]^2}\right\}.
    \end{align*}
\end{lem}
Applying the same approach as that used for the last two results in Section 2, we derive the following two analogous results.

\begin{lem}\label{lem:G-div-dn}
    Let $n$ and $d\geqslant2$ be positive integers satisfy $\gcd(n,d)=1$. Let $r$ be an integer with $d-n\leqslant r\leqslant (d-1)n$, and $n\equiv -r\pmod{d}$.
    For $k\in \{1,2,\dots, \frac{(d-1)n-r}{d}\}$, write \[\frac{(1-q^{(d-1)n})(q^d;q^d)_{k-1}(q^{2r};q^d)_{k-1}}{(q^r;q^d)_{k}(q^r;q^d)_{k-1}}
    =\frac{A_k(q)}{B_k(q)},\] where $A_k(q)$ and $B_k(q)$ are relatively prime polynomials in $q$. Then $B_k(q)$ is relatively prime to $1-q^n$.
\end{lem}
    
\begin{lem}\label{lem:q-integers-dn}
    Let $n$ and $d\geqslant2$ be positive integers satisfy $\gcd(n,d)=1$. Let $r$ be an integer with $d-n\leqslant r\leqslant (d-1)n$, and $n\equiv -r\pmod{d}$.
    For $k\in \{1,2,\dots, \frac{(d-1)n-r}{d}\}$, \[\frac{1-q^{(d-1)n}}{1-q^{dk}}=\frac{C_k(q)}{D_k(q)},\] where $C_k(q)$ and $D_k(q)$ are relatively prime polynomials in $q$. 
    Then $D_k(q)$ is relatively prime to $1-q^n$.
\end{lem}
We are now ready to complete the proof of Theorem~\ref{thm:main-2}.
\begin{proof}[Proof of Theorem~\ref{thm:main-2}]
Let \(F(m,k)\) and \(G(m,k)\) be as defined in Section~3.
Summing \eqref{eq:qz} with respect to \(m\) from \(0\) to \(((d-1)n-r)/d\), we obtain
\begin{align*}
    G&\left(\frac{(d-1)n+d-r}{d},k\right)\\
    &=\sum_{m=0}^{((d-1)n-r)/d}[dk-d+r]F(m,k-1)-\sum_{m=0}^{((d-1)n-r)/d}[dk-d+2r]F(m,k),
\end{align*}
where we have used the fact that $G(0,k)=0$.
Then, summing over $k$ from $1$ to $((d-1)n-r)/d$ and carrying out an iterative computation,
 observing that $F(m,((d-1)n-r)/d)=0$ for $m<((d-1)n-r)/d$, we obtain
\begin{align}\label{eq:FG-dn}
    \sum_{m=0}^{((d-1)n-r)/d}F(m,0)
    &=F\left(\frac{(d-1)n-r}{d},\frac{(d-1)n-r}{d}\right)\prod_{j=1}^{((d-1)n-r)/d}\frac{[dj-d+2r]}{[dj-d+r]}\nonumber\\
    &\hspace{1em}+\sum_{k=1}^{((d-1)n-r)/d}G\left(\frac{(d-1)n+d-r}{d},k\right)\prod_{j=1}^{k-1}\frac{[dj-d+2r]}{[dj-d+r]}\frac{1}{[dk-d+r]}.
\end{align}
By Lemma~\ref{lem:F-dn}, we conclude that, modulo $[n]\Phi_n(q)^4$,
\begin{align}\label{eq:F-dn-final}
    &F\left(\frac{(d-1)n-r}{d},\frac{(d-1)n-r}{d}\right)\prod_{j=1}^{((d-1)n-r)/d}\frac{[dj-d+2r]}{[dj-d+r]}\nonumber\\
    &\hspace{1em}=(-1)^{((d-1)n-r)/d}q^{(dn-n-r)(d-dn+n-3r)/2d}[2(d-1)n-r]\nonumber\\
    &\hspace{2em}\times\frac{(q^r;q^d)_{2((d-1)n-r)/d}}{(q^d;q^d)^2_{((d-1)n-r)/d}}\cdot
    \frac{(q^{2r};q^d)_{((d-1)n-r)/d}}{(q^d;q^d)_{((d-1)n-r)/d}}\nonumber\\
    &\hspace{1em}\equiv q^{r(r-(d-1)n)/d}\frac{(q^{2r};q^d)_{((d-1)n-r)/d}}{(q^d;q^d)_{((d-1)n-r)/d}}\nonumber\\
    &\hspace{2em}\times\left\{[(d-1)n]-([(d-1)n]^3+[(d-1)n]^4(1-q))
    \sum_{k=1}^{((d-1)n-r)/d}\frac{q^{kd}}{[kd]^2}\right\}.
\end{align}
For $k=1,2,\dots,((d-1)n-r)/d$, we see that 
\begin{align}\label{eq:G-dn}
    G\left(\frac{(d-1)n+d-r}{d},k\right)&=(-1)^{k-1}q^{d\binom{k}{2}-((d-1)n-r)(k-(d-2r)/d)}\nonumber\\
    &\hspace{1em}\times\frac{(q^r;q^d)^3_{((d-1)n-r)/d+1}(q^r;q^d)_{((d-1)n-r)/d+k}}
    {(1-q)^2(q^d;q^d)^3_{((d-1)n-r)/d}(q^d;q^d)_{((d-1)n-r)/d-k+1}(q^r;q^d)_k^2}\nonumber\\
    &\hspace{0.5em}=(-1)^{k-1}q^{d\binom{k}{2}-((d-1)n-r)(k-(d-2r)/d)}\nonumber\\
    &\hspace{1em}\times\frac{(1-q^{(d-1)n})^4(q^r;q^d)^4_{((d-1)n-r)/d}(q^{(d-1)n+d};q^d)_{k-1}}
    {(1-q)^2(q^d;q^d)^3_{((d-1)n-r)/d}(q^d;q^d)_{((d-1)n-r)/d-k+1}(q^r;q^d)^2_k}.
\end{align}
Making use of the fact that $q^n\equiv 1 \pmod{\Phi_n(q)}$, it is readily seen that
\begin{align*}
    (q^d;q^d)_{((d-1)n-r)/d-k+1}\equiv (-1)^{k-1}q^{(k-1)(dk+2r-2d)/2}\frac{(q^d;q^d)_{((d-1)n-r)/d}}{(q^r;q^d)_{k-1}}\pmod{\Phi_n(q)}
\end{align*}
and 
\begin{align*}
    \frac{(q^r;q^d)_{((d-1)n-r)/d}}{(q^d;q^d)_{((d-1)n-r)/d}}\equiv (-1)^{((d-1)n-r)/d}q^{(dn-n-r)(dn-n+r-d)/2d}\pmod{\Phi_n(q)}.
\end{align*}
In view of the above two $q$-congruences, it follows from \eqref{eq:G-dn} that, for $1\leqslant k\leqslant((d-1)n-r)/d$,
\begin{align*}
    G&\left(\frac{(d-1)n+d-r}{d},k\right)\\
    &\equiv q^{(kd-d-2dn+2n+2r)(d-dn+n)/d}\frac{[(d-1)n]^4(q^d;q^d)_{k}}{[dk][dk-d+r](q^r;q^d)_{k}}\pmod{\Phi_n(q)^5}.
\end{align*}
Furthermore, we obtain
\begin{align}\label{eq:G-dn-final}
    &G\left(\frac{(d-1)n+d-r}{d},k\right)\prod_{j=1}^{k-1}\frac{[dj-d+2r]}{[dj-d+r]}\cdot\frac{1}{[dk-d+r]}\nonumber\\
    &\hspace{1em}\equiv q^{(kd-d-2dn+2n+2r)(d-dn+n)/d}\frac{[(d-1)n]^4(q^d;q^d)_{k}(q^{2r};q^d)_{k-1}}{[dk][dk-d+r]^2(q^r;q^d)_{k}(q^r;q^d)_{k-1}}\pmod{\Phi_n(q)^5}.
\end{align}
By virtue of Lemma~\ref{lem:G-div-dn} and Lemma~\ref{lem:q-integers-dn}, it is easy to see that the right-hand side 
of \eqref{eq:G-dn-final} is congruent to 0 modulo $[n]$. Similar as before, we can obtain that $G(\tfrac{(d-1)n+d-r}{d},k)\equiv 0\pmod{[n]}$ from 
\eqref{eq:G-dn}, and so is the left-hand side of \eqref{eq:G-dn-final}. Namely, the $q$-congruence \eqref{eq:G-dn-final} holds modulo $[n]$.
It follows from the identity $\operatorname{lcm}([n],\Phi_n(q)^5)=[n]\Phi_n(q)^4$ that the $q$-congruence \eqref{eq:G-dn-final} holds modulo $[n]\Phi_n(q)^4$.\par
Combining \eqref{eq:FG-dn}, \eqref{eq:F-dn-final} and the case modulo $[n]\Phi_n(q)^4$, we immediately complete the proof by using the identity \eqref{eq:iden}.
\end{proof}

\vskip 5mm\noindent{\bf  Declaration of competing interest.} There is no competing interest.

\vskip 5mm \noindent{\bf Data Availability Statements.}
Data sharing not applicable to this article as no datasets were generated or analysed during the current study.

\vskip 5mm \noindent{\bf Acknowledgments.} The second author was partially supported by the National Natural Science Foundation of China (No. 12401435), the Natural Science
 Foundation of Inner Mongolia, China (No. 2024MS01017), the First-Class Disciplines Project, Inner Mongolia Autonomous Region, China (No. YLXKZX-NSD-014) and Program for Innovative Research Team in Universities of Inner Mongolia Autonomous Region (No. NMGIRT2414).

\end{document}